\documentclass[12pt]{article}
\usepackage{graphicx}
\usepackage{amsmath}
\usepackage{amssymb}
\usepackage{theorem}

      \usepackage{hyperref}

\numberwithin{equation}{section}

\newtheorem{thm}{Theorem}[section]
\newtheorem{lemma}[thm]{Lemma}
\newtheorem{prop}[thm]{Proposition}
\newtheorem{cor}[thm]{Corollary}
{\theorembodyfont{\rmfamily}

\newtheorem{rmk}[thm]{Remark}
}

\newcommand{\qed}{\hfill \mbox{\raggedright \rule{.07in}{.1in}}}
 
\newenvironment{proof}{\vspace{1ex}\noindent{\bf
Proof}\hspace{0.5em}}{\hfill\qed\vspace{1ex}}
\newenvironment{pfof}[1]{\vspace{1ex}\noindent{\bf Proof of
#1}\hspace{0.5em}}{\hfill\qed\vspace{1ex}}

\newcommand{\R}{{\mathbb R}}
\newcommand{\C}{{\mathbb C}}
\newcommand{\Z}{{\mathbb Z}}
\newcommand{\N}{{\mathbb N}}

\newcommand{\cC}{{\mathcal C}}
\newcommand{\cD}{{\mathcal D}}
\newcommand{\cH}{{\mathcal H}}
\newcommand{\cP}{{\mathcal P}}
\newcommand{\tcP}{{\widetilde{\mathcal P}}}
\newcommand{\hB}{{\widehat B}}
\newcommand{\hcD}{{\widehat {\mathcal D}}}
\newcommand{\tcD}{{\widetilde {\mathcal D}}}

\newcommand{\eps}{\varepsilon}
\newcommand{\Leb}{\operatorname{Leb}}
 \newcommand{\type}{\operatorname{type}}
 \newcommand{\range}{\operatorname{range}}
\newcommand{\diam}{\operatorname{diam}}

\title{Good inducing schemes for uniformly hyperbolic flows,
and applications to exponential decay of correlations}

\author{
Ian Melbourne \thanks{Mathematics Institute, University of Warwick, Coventry, CV4 7AL, UK}
\and 
Paulo Varandas \thanks{CMUP and Departamento de Matem\'atica, Universidade Federal da Bahia, 40170-110 Salvador, Brazil}
}

\date{21 February 2023, updated 17 May 2023}

\begin{document}

 \maketitle

 \begin{abstract}
Given an Axiom~A attractor for a $C^{1+\alpha}$ flow ($\alpha>0$), we construct a countable Markov extension with exponential return times in such a way that the inducing set is a smoothly embedded unstable disk. This avoids technical issues concerning irregularity of boundaries of Markov partition elements and enables an elementary approach to certain questions involving exponential decay of correlations for SRB measures.
 \end{abstract}

 \section{Introduction} 
 \label{sec:intro}

Statistical properties~\cite{Bowen75,Ruelle78,Sinai72} of Anosov and Axiom~A diffeomorphisms~\cite{Anosov67,Smale67} were developed extensively in the 1970s.
Key tools were the construction of finite Markov partitions~\cite{Bowen70,Sinai68}
and the spectral properties of transfer operators~\cite{Ruelle68}.
In particular, ergodic invariant probability measures were constructed corresponding to any H\"older potential; moreover, it was shown that hyperbolic basic sets for Axiom~A diffeomorphisms are always exponentially mixing up to a finite cycle for such measures, see for example~\cite{Bowen75,ParryPollicott90,Ruelle78}.

Still in the 1970s, finite Markov partitions were constructed~\cite{Bowen73,Ratner73b} 
for Anosov and Axiom~A flows. This allows us to model each hyperbolic basic set as a suspension flow over a subshift of finite type, enabling the study of thermodynamic formalism (see e.g.~\cite{BowenRuelle75}) and statistical properties 
(see e.g.~\cite{DenkerPhilipp84,Ratner73}).

However, rates of mixing for Axiom~A flows are still poorly understood. 
By~\cite{Pollicott84,Ruelle83}, mixing Axiom~A flows can mix arbitrarily 
slowly. 
Although there has been important progress starting with~\cite{Chernov98,Dolgopyat98a,Liverani04},
it remains an open question whether mixing Anosov flows have exponential decay of correlations~\cite{BowenRuelle75}. 
Very recently, this question was answered positively~\cite{TsujiiZhang23} in the case of $C^\infty$ three-dimensional flows.

It turns out that using finite Markov partitions for flows raises technical issues due to the irregularity of their boundaries~\cite{ABV18,ButterleyWar20,Stoyanov11}. 
Even in the discrete-time setting, it is known that the boundaries of elements of a finite Markov partition need not be smooth~\cite{Bowen78}.
In this paper, we propose using the approach of~\cite{Young98} to circumvent such issues at least in the case of SRB measures. In particular, we show that 
\begin{quote}
\emph{Any attractor for an Axiom~A flow can be modelled by a suspension flow over a full branch countable Markov extension where the inducing set is a smoothly embedded unstable disk. The roof function, though unbounded, has exponential tails.}
\end{quote}
A precise statement is given in Theorem~\ref{thm:induce} below.

\begin{rmk} The approach of Young towers~\cite{Young98} has proved to be highly effective for studying discrete-time examples like planar dispersing billiards and H\'enon-like attractors where suitable Markov partitions are not available. However, as shown in the current paper, there can be advantages (at least in continuous time) to working with countable Markov extensions even when there is a well-developed theory of finite Markov partitions. The extra flexibility of Markov extensions can be used not only to construct the extension but to ensure good regularity properties of the partition elements.
\end{rmk}

As a consequence of Theorem~\ref{thm:induce}, we obtain an elementary proof of the following result:

\begin{thm} \label{thm:A}
 Suppose that $\Lambda$ is an Axiom~A attractor with SRB measure $\mu$ for 
a $C^{1+}$ flow $\phi_t$ with 
$C^{1+}$ stable holonomies\footnote{$C^{1+}$ means $C^r$ for some $r>1$.}
and such that the stable and unstable bundles are not jointly integrable.
Then for all H\"older observables $v,\,w:\Lambda\to\R$,
there exist constants $c,C>0$ such that
\[
\Big|\int_\Lambda v\,w\circ \phi_t\,d\mu
-\int_\Lambda v\,d\mu \int_\Lambda w\,d\mu\Big|\le C e^{-c t}
\quad\text{for all $t>0$}. 
\]
\end{thm}

\begin{rmk}
Joint nonintegrability holds for an open and dense set of Axiom~A flows and their attractors, see~\cite{FMT07} and references therein. 
It implies mixing and is equivalent to mixing for codimension one Anosov flows. It is conjectured to be equivalent to mixing for Anosov flows~\cite{Plante72}.
\end{rmk}

\begin{rmk}
(a) In the case when the unstable direction is one-dimensional and the stable holonomies are $C^2$, this result is due to~\cite{BaladiVallee05,AGY06,ABV16,ABV18}.
In particular, using the fact that stable bunching is a robust sufficient condition for smoothness of stable holonomies together with the robustness of joint nonintegrability,~\cite{ABV16} constructed the first robust examples of 
Axiom~A flows with exponential decay of correlations. 
The smoothness condition on stable holonomies was relaxed from $C^2$ to $C^{1+}$ in~\cite{AraujoM16} extending the class of examples in~\cite{ABV16}.
This class of examples is extended further by Theorem~\ref{thm:A} 
with the removal of the one-dimensionality restriction on unstable manifolds.

\vspace{1ex} \noindent
(b) There is no restriction on the dimension of unstable manifolds in~\cite{AGY06}, and it is not surprising that the smoothness assumption on stable holonomies can also be relaxed as in~\cite{AraujoM16}. 
However, there is a crucial hypothesis in~\cite{AGY06} on the regularity of the inducing set in the unstable direction which is nontrivial in higher dimensions.

Theorem~\ref{thm:A} is stated in the special case of Anosov flows in~\cite{ButterleyWar20}.
In~\cite[Appendix]{ButterleyWar20} it is argued that at least in the Anosov case the Markov partitions of~\cite{Ratner73b} are sufficiently regular that the methods in~\cite{AGY06} can be pushed through. 
In~\cite{ABV18}, a sketch is given of how to prove Theorem~\ref{thm:A} also in the Axiom A case, but the details are not fully worked out.

As mentioned, our approach in this paper completely bypasses such issues since our inducing set is a smoothly embedded unstable disk.
Moreover, our method works equally well for Anosov flows and Axiom~A attractors.
As a consequence, we recover the examples in~\cite{ButterleyWar20}, in particular that codimension one volume-preserving mixing $C^{1+}$ Anosov flows are exponentially mixing in dimension four and higher.
\end{rmk}

The remainder of the paper is organised as follows.
In Section~\ref{sec:induce}, we state precisely and prove our result on good inducing for attractors of Axiom~A flows.
In Section~\ref{sec:expdecay}, we prove a result on exponential mixing for a class of skew product Axiom~A flows, extending/combining the results in~\cite{AraujoM16,AGY06}.
In Section~\ref{sec:proof}, we complete the proof of Theorem~\ref{thm:A}.

\section{Good inducing for attractors of Axiom~A flows}
\label{sec:induce}

Let $\phi_t: M\to  M$ be a $C^{1+}$ flow
defined on a compact Riemannian manifold $(M,d_M)$,
 and let 
$\Lambda\subset M$ be a closed $\phi_t$-invariant subset.
We assume that $\Lambda$ is an attracting transitive uniformly hyperbolic set with adapted norm and that $\Lambda$ is not a single trajectory.
In particular, there is a continuous $D\phi_t$-invariant splitting $T_\Lambda M=E^s\oplus E^c\oplus  E^u$ where $E^c$ is the one-dimensional central direction tangent to the flow, and
there exists $\lambda\in(0,1)$ such that
$|D\phi_t v|\le\lambda^t|v|$ for all $v\in E^s$, $t\ge 1$;
$|D\phi_{-t} v|\le\lambda^t|v|$ for all $v\in E^u$, $t\ge 1$. 
Since the time-$s$ map $\phi_s:\Lambda\to\Lambda$ is ergodic for all but countably many choices of $s\in\R$~\cite{PughShub71},
we can scale time slightly if necessary so that $\phi_{-1}:\Lambda\to\Lambda$ is transitive.
Then there exists $p\in \Lambda$ such that $\bigcup_{i\ge1}\phi_{-i}p$ is dense in $\Lambda$.

We can define (local) stable disks $W^s_\delta(y)=\{z\in W^s(y):d_M(y,z)<\delta\}$ for $\delta>0$ sufficiently small for all $y\in \Lambda$.
Define 
local centre-stable disks $W^{cs}_\delta(y)=\bigcup_{|t|<\delta}\phi_t W^s_\delta(y)$.
Let $\Leb$ and $d$ denote induced Lebesgue measure and induced distance on local unstable manifolds.
It is convenient to define local unstable disks 
$W^u_\delta(y)=\{z\in W^u(y):d(y,z)<\delta\}$ using the induced distance.

For $\delta_0$ small, define $\cD=W^u_{\delta_0}(p)$
and $\hcD=\bigcup_{x\in \cD}W^{cs}_{\delta_0}(x)$.
Define $\pi:\hcD\to\cD$ such that $\pi| W^{cs}_{\delta_0}(x)\equiv x$.
Whenever $\phi_ny\in \hcD$, we set $g_ny=\pi(\phi_ny)$.

We are now in a position to give a precise description of our inducing scheme.

\begin{thm} \label{thm:induce}
There exists an open unstable disk $Y=W^u_\delta(p)\subset \cD$ (for some $\delta\in (0,\delta_0)$) and a discrete return time function $R:Y\to \Z^+\cup\{\infty\}$ such that 
\begin{itemize}
\item[(i)] $\Leb(R>n)=O(\gamma^n)$ for some $\gamma\in(0,1)$;
\item[(ii)]  Each connected component of $\{R=n\}$ is mapped by $\phi_n$ into $\hcD$ and
mapped homeomorphically by $g_n$ onto $Y$.
\end{itemize}
\end{thm}

\begin{rmk} \label{rmk:induce}
Let $\cP$ be the partition of $Y$ consisting of connected components of $\{R=n\}$ for $n\ge1$. (It follows from Theorem~\ref{thm:induce}(i) that $\cP$ is a partition of $Y\bmod0$.) Define $F:Y\to Y$, $F=g_R=\pi\circ \phi_R$.
Note that $F$ is locally the composition of a time-$R$ map $\phi_R$ (where $R$ is constant on each partition element) with a centre-stable holonomy.
Since centre-stable holonomies are H\"older continuous,
it follows that $F$ maps partition elements $U\in\cP$ homeomorphically onto $Y$ 
and that $F|_U:U\to Y$ is a bi-H\"older bijection.
If moreover, the centre-stable holonomies are $C^1$, then the partition elements are diffeomorphic to disks.
\end{rmk}

In the remainder of this section, we prove Theorem~\ref{thm:induce}.
Our proof is essentially the same as in~\cite[Section~6]{Young98} for Axiom~A diffeomorphisms, but we closely follow the treatment in~\cite{AlvesLuzzattoPinheiro05} which provides many of the details of arguments sketched in~\cite{Young98}.

\paragraph{Choice of constants}
Choose $\delta_0>0$ so that 
the following
bounded distortion property holds: there exists $C_1\ge1$ so that
\begin{equation}\label{eq:bdd}
\frac{|\det D\phi_n(x)|E^u|}{|\det D\phi_n(y)|E^u|} \leq C_1
\end{equation}
for every $n\ge 1$ and all $x,y\in \Lambda$ with $\phi_nx,\phi_ny$ in the same unstable disk such that $d(\phi_jx,\phi_jy)<4\delta_0$ for all $0\le j \le n$. 

By standard results about stable holonomies, $\pi$ is absolutely continuous and $C^\alpha$ for some $\alpha\in(0,1)$ when restricted to unstable disks in $\hcD$.
For $\delta_0$ sufficiently small, there exists $C_2$, $C_3\ge1$ such that
\begin{equation}\label{eq:bddpi}
C_2^{-1}\le \frac{\Leb(\pi(E))}{\Leb(E)} \leq C_2
\end{equation}
for all Lebesgue-measurable subset $E\subset W^u_{\delta_0}(y)\cap\hcD$ and all $y\in\Lambda$,
and
\begin{equation} \label{eq:C2}
d(\pi x,\pi y)\le C_3 d(x,y)^\alpha
\end{equation}
for all $x,y\in \hcD$ with $x,y$ in the same unstable disk such that $d(x,y)<4\delta_0$.

Let $d_u=\dim E^u$ and
fix $L\ge3$ so that 
\begin{equation}\label{eq:L}
C_1C_2^2\frac{2^{d_u}-1}{(L-1)^{d_u}}<\frac14 .
\end{equation}

By the local product structure, there exists $\delta_1\in(0,\delta_0)$ such that
$W^{cs}_{\delta_0}(x)\cap W^u_{\delta_0}(y)$ consists of precisely one point for
all $x,y\in\Lambda$ with $d_M(x,y)<4\delta_1$. Similarly, there exists $\delta\in(0,\delta_1)$ such that
$W^{cs}_{\delta_1}(x)\cap W^u_{\delta_1}(y)$ consists of precisely one point for
all $x,y\in\Lambda$ with $d_M(x,y)<(L+1)\delta$.
Moreover, since local unstable/stable disks have bounded curvature, the intersection point $z\in W^{cs}_{\delta_1}(x)\cap W^u_{\delta_1}(y)$ satisfies $d(z,y)\le C_4d_M(x,y)$ where $C_4\ge1$ is a constant.
Shrink $\delta>0$ if necessary so that $C_3(3\delta)^\alpha<\frac12\delta_0$ and $C_4(L+1)\delta<\delta_0$.
Choose
$N_1\ge1$ such that $\bigcup_{i=1}^{N_1} \phi_{-i}p$ is $\delta$-dense in $\Lambda$.

\paragraph{Construction of the partition}

We consider various small neighbourhoods $\cD_c=W^u_{c\delta}(p)$ with $c\in\{1,2,L-1,L\}$.
Define $\tcD_c=\bigcup_{x\in \cD_c}W^{cs}_{\delta_1}(x)$.

Take $Y=\cD_1$.
Define a partition $\{I_k:k\ge1\}$ of $\cD_2\setminus \cD_1$,
\[
I_k=\big\{y\in \cD_2: \delta(1+\lambda^{\alpha k}) \le d(y,p) < \delta(1+\lambda^{\alpha(k-1)})\big\}.
\]

Fix $\eps>0$ small (as stipulated in Propositions~\ref{prop:eps} and~\ref{prop:eps2} and Lemma~\ref{lem:keyfact} below). 
We define sets $Y_n$ and functions 
$t_n:Y_n\to\N$, and $R:Y\to \Z^+$ inductively, with
$Y_n=\{R> n\}$.
Define $Y_0=Y$ and $t_0\equiv 0$. 
Inductively, suppose that $Y_{n-1}=Y\setminus\{R<n\}$ and that $t_{n-1}:Y_{n-1}\to\N$ is given. Write $Y_{n-1}=A_{n-1}\,\dot\cup\,B_{n-1}$ where
\[
A_{n-1}=\{t_{n-1}=0\}, \qquad
B_{n-1}=\{t_{n-1}\ge1\}.
\]
Consider the neighbourhood 
\[
A^{(\eps)}_{n-1}=\big\{y\in Y_{n-1}: d(\phi_ny,\phi_nA_{n-1})<\eps\big\}
\]
 of the set $A_{n-1}$.
Define $U_{nj}^L$, $j\ge1$, to be the connected components of $A^{(\eps)}_{n-1}\cap \phi_{-n}\tcD_L$ that are mapped inside $\tcD_L$ by $\phi_n$ and mapped homeomorphically onto $\cD_L$ by $g_n$. 
Let
\[
U_{nj}^c=U_{nj}^L\cap g_n^{-1}\cD_c \quad\text{for $c=1,2,L-1$.}
\]
Define $R|U_{nj}^1=n$ for each $U_{nj}^1$ and take $Y_n=Y_{n-1}\setminus \bigcup_j U_{nj}^1$.
Finally, define $t_n:Y_n\to\N$ as
\[
t_n(y) =
\begin{cases}
\begin{array}{cl}
k, & y\in \bigcup_j U_{nj}^2 \text{ and } g_ny\in I_k \, \text{for some $k\ge 1$}\\
0, & y\in A_{n-1} \setminus \bigcup_j U_{nj}^2 \\
t_{n-1}(y)-1, & y\in B_{n-1} \setminus \bigcup_j U_{nj}^2
\end{array}
\end{cases}
\]
and take $A_n=\{t_n=0\}$, 
$B_n=\{t_n\ge1\}$ and $Y_n = A_n\,\dot\cup\,B_n$.

\begin{rmk} By construction, property~(ii) of Theorem~\ref{thm:induce} is satisfied.
 It remains to verify that
$\Leb(R>n)$ decays exponentially.
\end{rmk}

\paragraph{Visualisation of $B_n$.}
The set $B_n$ is a disjoint union $B_n=\bigcup_{m=1}^n C_n(m)$ where $C_n(m)$ is a disjoint union of \emph{collars} around each component of $\{R= m\}$.
Each collar in $C_n(m)$ is homeomorphic under $g_m$ to $\bigcup_{k\ge n-m+1}I_k$ with outer ring homeomorphic under $g_m$ to $I_{n-m+1}$,
and the union of outer rings is the set $\{t_n=1\}$.
This picture presupposes Proposition~\ref{prop:eps} below which guarantees that each new generation of collars $C_n(n)$ does not intersect the set $\bigcup_{1\le m\le n-1}C_{n-1}(m)$ of collars in the previous generation.

\begin{prop} \label{prop:eps}
Choose $\eps<(C_3^{-1}\delta)^{1/\alpha}$ sufficiently small that $W^u_\eps(x)\subset\hcD$ for all $x\in\tcD_L$. Then
 $\bigcup_jU_{nj}^{L-1}\subset A_{n-1}$ for all $n\ge1$.
\end{prop}

\begin{proof}
We argue by contradiction.
There is nothing to prove for $n=1$. Let $n\ge2$ be least such that the result fails and choose $j$ such that $U_{nj}^{L-1}$ intersects $B_{n-1}$.
Then either (i)
$U_{nj}^{L-1}\subset B_{n-1}$, or
(ii)
$U_{nj}^{L-1}$ intersects $\partial A_{n-1}$.

\vspace{1ex}
In case (i), choose 
$x\in U_{nj}^{L-1}$ (so in particular $\phi_nx\in\hcD$) with $g_nx=p$. 
Since $U_{nj}^{L-1}\subset U_{nj}^L\subset A_{n-1}^{(\eps)}$, there exists $y\in A_{n-1}$ with $d(\phi_nx,\phi_ny)<\eps$. 
In particular, $\phi_ny\in\hcD$ so $g_ny$ is well-defined.
Note that $x\in U_{nj}^{L-1}$ and $y\not\in U_{nj}^{L-1}$ since $U_{nj}^{L-1}\subset B_{n-1}$.
Hence the geodesic $\ell$ in $\cD$ joining $g_nx$ and $g_ny$ intersects $g_n\partial U_{nj}^{L-1}$.
Choose
$z\in \partial U_{nj}^{L-1}\cap g_n^{-1}\ell$.
Since $g_n=\pi\circ \phi_n$, it follows from~\eqref{eq:C2} that
\[
\delta<(L-1)\delta=d(g_nx,g_nz)< d(g_nx,g_ny)\le C_3
d(\phi_nx,\phi_ny)^\alpha<C_3\eps^\alpha<\delta
\]
which is a contradiction. This rules out case~(i).

\vspace{1ex}
In case~(ii), choose $x\in U_{nj}^{L-1}\cap \partial A_{n-1}$.
We show below that there exists $y\in \partial A_{n-1}^{(\eps)}$ such that $d(\phi_nx,\phi_ny)\le\eps$.
In particular, $g_nx$ and $g_ny$ are well-defined and $d(g_nx,g_ny)\le C_3\eps^\alpha<\delta$.
Since $U_{nj}^L\subset A_{n-1}^{(\eps)}$, we have that
$y\not\in U_{nj}^L$.
It follows that $g_nx\in \cD_{L-1}$ while $g_ny\not\in \cD_L$.
Hence $d(g_nx,g_ny)\ge\delta$ which is the desired contradiction. 

It remains to verify 
that there exists $y\in \partial A_{n-1}^{(\eps)}$ such that $d(\phi_nx,\phi_ny)\le\eps$.
Since~$n$ is least, $B_{n-1}$ is a disjoint union of collars as described in the visualisation above.
Hence there exists a collar $Q\subset C_{n-1}(n-k)$ 
 intersected by $U_{nj}^{L-1}$ for some $1\le k<n$ such that $x$ lies in the outer boundary $\partial_oQ$ of $Q$. 
Note that $\partial_o Q= \partial A_{n-1}\cap Q$. 
Let $D$ denote the disk enclosed by $\partial_oQ$ and let 
\[
S=D\cap\partial (\phi_{-n}B_\eps(\phi_n\partial D)).
\]
We claim that $S\neq\emptyset$ and $S\subset Q$.
Then $S$ is a $(\dim Y-1)$-dimensional sphere contained in $\partial A_{n-1}^{(\eps)}$ and there exists $y\in S$ with the desired properties.
(The point of the claim is that $S$ lies entirely in $Y_{n-1}$.)

Note that $g_{n-k}$ maps $Q$ homeomorphically onto the set $J=\bigcup_{i\ge k}I_i$ which is an annulus of radial thickness $\delta\lambda^{\alpha k}$. 
By~\eqref{eq:C2}, $\phi_{n-k}$
maps $Q$ homeomorphically onto a set $\tilde J=\pi^{-1}J$ of radial thickness at least $(C_3^{-1}\delta\lambda^{\alpha k})^{1/\alpha}
= (C_3^{-1}\delta)^{1/\alpha}\lambda^k$. 

Moreover, 
$\phi_k(\tilde J\cap \phi_{n-k}A_{n-1}^{(\eps)})\subset \phi_nA_{n-1}^{(\eps)}$ is contained in the set of points within $d$-distance~$\eps$ of $\phi_n\partial A_{n-1}^{(\eps)}$, so 
by definition of $\lambda$ we have that
$\tilde J\cap \phi_{n-k}A_{n-1}^{(\eps)}$ is contained in the set of points within $d$-distance $\eps\lambda^k$ of the outer boundary of $\tilde J$. Since $\eps<(C_3^{-1}\delta)^{1/\alpha}$, 
we obtain that
$\tilde J\cap \phi_{n-k}\partial A_{n-1}^{(\eps)}$ is homeomorphic to a $(\dim Y-1)$-dimensional
sphere contained entirely inside $\tilde J$.
Hence
$S=Q\cap \partial A_{n-1}^{(\eps)}$ is homeomorphic to a $(\dim Y-1)$-dimensional
sphere contained entirely inside $Q$, as required.~
\end{proof}

\begin{prop} \label{prop:eps2}
Choose $\eps<\big\{C_3^{-1}\delta(\lambda^{-\alpha}-1)\big\}^{1/\alpha}$.
Then for all $n\ge1$,
\begin{itemize}
\item[(a)]
$A_{n-1}^{(\eps)}\subset \{y\in Y_{n-1}:t_{n-1}(y)\le 1\}$ for all $n\ge1$.
\item[(b)]
$\phi_{-n}W_\eps^u(\phi_nx)\subset A_{n-1}^{(\eps)}$ for all $x\in A_{n-1}$.
\end{itemize}
\end{prop}

\begin{proof}
(a) Suppose that $t_{n-1}(y)>1$.
Then there exists a collar in $C_{n-1}(n-k)$ containing $y$.
Let $Q$ denote the outer ring of the collar with outer boundary $Q_1$ and inner boundary $Q_2$. Then $t_{n-1}|Q\equiv1$ and $t_{n-1}(y)>1$, so $y$ lies inside the region bounded by $Q_2$.

Suppose for contradiction that $y\in A_{n-1}^{(\eps)}$. Then we can choose $x\in A_{n-1}$ with
$d(\phi_nx,\phi_ny)<\eps$. Let $\ell$ be the geodesic in $W^u_\eps(\phi_nx)$ connecting $\phi_nx$ to $\phi_ny$
and define $q_j\in Q_j\cap \phi_{-n}\ell$ for $j=1,2$.

Recall that $Q$ is homeomorphic under $g_{n-k}$ to $I_k$.
Moreover, $g_{n-k}q_j$ lie in distinct components of the boundary of $I_k$,
so 
\[
d(g_{n-k}q_1,g_{n-k}q_2)\ge\delta(\lambda^{\alpha(k-1)}-\lambda^{\alpha k})=\delta(\lambda^{-\alpha}-1)\lambda^{\alpha k}.
\]
Hence
\[
\begin{split}
d(\phi_nq_1,\phi_nq_2)
& \ge \lambda^{-k} d(\phi_{n-k}q_1,\phi_{n-k}q_2) 
\\ & \ge \lambda^{-k}\big\{C_3^{-1}
d(g_{n-k}q_1,g_{n-k}q_2)\big\}^{1/\alpha}\ge
\big\{C_3^{-1}\delta(\lambda^{-\alpha}-1)\big\}^{1/\alpha}
>\eps.
\end{split}
\]
But $d(\phi_nq_1,\phi_nq_2)\le d(\phi_ny,\phi_nx)<\eps$ so we obtain the desired contradiction.

\vspace{1ex}
(b) 
Let $x\in A_{n-1}$ and $y\in \phi_{-n}W^u_\eps(\phi_nx)$. 
Note that $y\in A_{n-1}^{(\eps)}$ if and only if $y\in Y_{n-1}$. Hence we
must show that $y\in Y_{n-1}$. If not, then there
exists $k\ge1$ such that
$y\in \{R=n-k\}$. Define $Q\subset C_{n-1}(n-k)$ to be the outer ring of the corresponding collar. Choosing $q_1$ and $q_2$ as in part~(a) we again obtain a contradiction.
\end{proof}

\begin{lemma} \label{lem:facts}
There exists $a_1>0$ such that for all $n\ge1$,
\begin{itemize}
\item[(a)] $\Leb(B_{n-1}\cap A_n)\ge a_1\Leb(B_{n-1})$.
\item[(b)] $\Leb(A_{n-1}\cap B_n)\le \tfrac14\Leb(A_{n-1})$.
\item[(c)] $\Leb(A_{n-1}\cap \{R=n\})\le \tfrac14\Leb(A_{n-1})$.
\end{itemize}
\end{lemma}

\begin{proof}
(a)
Let $y\in B_{n-1}$. By Proposition~\ref{prop:eps}, $y\not\in \bigcup_j U_{nj}^{L-1}$  so in particular $y\in Y_n$.
Note that 
$t_n(y)=0$ if and only if $t_{n-1}(y)=1$.
Hence $B_{n-1}\cap A_n=\{t_{n-1}=1\}$.

Now let $Q\subset C_{n-1}(n-k)\subset B_{n-1}$ be a collar ($1\le k\le n$) with outer ring
$Q\cap A_n=Q\cap\{t_{n-1}=1\}$.
Then
$g_{n-k}=\pi\circ\phi_{n-k}$ maps $Q$ homeomorphically onto $\bigcup_{i\ge k}I_i$ and
$Q\cap\{t_{n-1}=1\}$ homeomorphically onto $I_k$. 
Let $d_u=\dim E^u$. By~\eqref{eq:bdd} and~\eqref{eq:bddpi},
\[
\begin{split}
\frac{\Leb(Q)}{\Leb(Q\cap A_n)} & =
\frac{\Leb(Q)}{\Leb(Q\cap\{t_{n-1}=1\})}
\le C_1\frac{\Leb(\phi_{n-k}Q)}{\Leb(\phi_{n-k}(Q\cap\{t_{n-1}=1\}))}
 \\ & \le C_1C_2^2\frac{\Leb(\bigcup_{i\ge k}I_i)}{\Leb(I_k)}
= C_1C_2^2D(d_u,\lambda^\alpha,k)
\end{split}
\]
where $D(d_u,\lambda,k)=\dfrac{(1+\lambda^{k-1})^{d_u}-1}{(1+\lambda^{k-1})^{d_u}-(1+\lambda^k)^{d_u}}$. Since $\lim_{k\to\infty}D(d_u,\lambda,k)=(1-\lambda)^{-1}$, we obtain
that
$\Leb(Q)\le C_1C_2^2D\Leb(Q\cap A_n)$ where 
$D=\sup_{k\ge1}D(d_u,\lambda^\alpha,k)$
is a constant depending only on $d_u$ and $\lambda^\alpha$. Summing over collars $Q$, it follows that
$\Leb (B_{n-1})\le C_1C_2^2D \Leb(B_{n-1}\cap A_n)$.
\\[.75ex]
(b)
By Proposition~\ref{prop:eps},
$U_{nj}^2\subset U_{nj}^{L-1}\subset A_{n-1}$ for each $j$.
It follows that
$A_{n-1}\cap B_n=\bigcup_j U_{nj}^2\setminus U_{nj}^1$.
By~\eqref{eq:bdd},~\eqref{eq:bddpi} and~\eqref{eq:L},
\[
\frac{\Leb(U_{nj}^2\setminus U_{nj}^1)}
{\Leb(U_{nj}^{L-1})}\le C_1 C_2^2
\frac{\Leb(\cD_2\setminus \cD_1)}
{\Leb(\cD_{L-1})}=C_1C_2^2\frac{2^{d_u}-1}{(L-1)^{d_u}}<\frac14.
\]
Hence
\[
\frac{\Leb(A_{n-1}\cap B_n)}{\Leb(A_{n-1})}\le 
\frac{\sum_j\Leb(U_{nj}^2\setminus U_{nj}^1)}
{\sum_j\Leb(U_{nj}^{L-1})}
<\frac14.
\]
(c)
Proceeding as in part~(b) with $U_{nj}^2\setminus U_{nj}^1$ replaced by $U_{nj}^1$, leads to the estimate
\[
\frac{\Leb(A_{n-1}\cap \{R=n\})}{\Leb(A_{n-1})}\le 
\frac{\sum_j\Leb(U_{nj}^1)}
{\sum_j\Leb(U_{nj}^{L-1})}\le \frac{C_1C_2^2}{(L-1)^{d_u}}<\frac14.
\]

\vspace{-5ex}
\end{proof}

\begin{cor} \label{cor:facts}
For all $n\ge1$,
\begin{itemize}
\item[(a)] $\Leb(A_{n-1}\cap A_n)\ge \tfrac12\Leb(A_{n-1})$.
\item[(b)] $\Leb(B_{n-1}\cap B_n)\le (1-a_1)\Leb(B_{n-1})$.
\item[(c)] $\Leb(B_n)\le \frac14\Leb(A_{n-1})+(1-a_1)\Leb(B_{n-1})$.
\item[(d)] $\Leb(A_n)\ge \frac12\Leb(A_{n-1})+a_1\Leb(B_{n-1})$.
\end{itemize}
\end{cor}

\begin{proof} 
Recall that $A_{n-1}\subset Y_{n-1}=
Y_n\,\dot\cup\,\{R=n\} =A_n\,\dot\cup\, B_n\,\dot\cup\,\{R=n\}$.
Hence by Lemma~\ref{lem:facts}(b,c),
\[
\begin{split}
\Leb(A_{n-1}) & =\Leb(A_{n-1}\cap A_n)
+\Leb(A_{n-1}\cap B_n)
+\Leb(A_{n-1}\cap \{R=n\})
\\ & \le \Leb(A_{n-1}\cap A_n)+\tfrac12\Leb(A_{n-1}),
\end{split}
\]
proving (a).
Similarly, 
by Lemma~\ref{lem:facts}(a),
\[
\begin{split}
\Leb(B_{n-1}) & =\Leb(B_{n-1}\cap A_n)
+\Leb(B_{n-1}\cap B_n)
+\Leb(B_{n-1}\cap \{R=n\})
\\ & \ge a_1\Leb(B_{n-1})+\Leb(B_{n-1}\cap B_n),
\end{split}
\]
proving (b).

Next, recall that 
$B_n=B_n\cap Y_{n-1}= B_n\cap\big(A_{n-1}\,\dot\cup\, B_{n-1}\big)$.
Hence part~(c) follows from Lemma~\ref{lem:facts}(b) and
part~(b).
Similarly, $A_n=A_n\cap\big(A_{n-1}\,\dot\cup\, B_{n-1}\big)$ and part~(d) follows from
Lemma~\ref{lem:facts}(a) and part~(a).
\end{proof}

\begin{cor} \label{cor:facts2}
There exists $a_0>0$ such that
\(
\Leb(B_n)\le a_0\Leb(A_n)
\)
for all $n\ge0$.
\end{cor}

\begin{proof} 
Let \(
a_0=\dfrac{2+a_1}{2a_1}.
\)
We prove the result by induction.
The case $n=0$ is trivial since $B_0=\emptyset$.
For the induction step from $n-1$ to $n$,
we consider separately the cases $\Leb(B_{n-1})>\frac{1}{2a_1}\Leb(A_{n-1})$
and $\Leb(B_{n-1})\le \frac{1}{2a_1}\Leb(A_{n-1})$.

Suppose first that $\Leb(B_{n-1})>\frac{1}{2a_1}\Leb(A_{n-1})$.
By Corollary~\ref{cor:facts}(c),
\[
\Leb(B_n)
< \big\{\tfrac12 a_1+(1-a_1)\big\}\Leb(B_{n-1})
= (1-\tfrac12 a_1)\Leb(B_{n-1})
< \Leb(B_{n-1}).
\]
By Corollary~\ref{cor:facts}(d),
\[
\Leb(A_n)> (\tfrac12+a_1\tfrac{1}{2a_1})\Leb(A_{n-1})=\Leb(A_{n-1}).
\]
Hence by the induction hypothesis,
\[
\Leb(B_n)<\Leb(B_{n-1})\le a_0\Leb(A_{n-1})<a_0\Leb(A_n),
\]
establishing the result at time $n$.

Finally, suppose that $\Leb(B_{n-1})\le\frac{1}{2a_1}\Leb(A_{n-1})$.
By Corollary~\ref{cor:facts}(a,c),
\[
\begin{split}
\Leb(B_n) & \le \tfrac14\Leb(A_{n-1})+\Leb(B_{n-1})
\le (\tfrac14+\tfrac{1}{2a_1})\Leb(A_{n-1})
\\ & \le (\tfrac12+\tfrac{1}{a_1})\Leb(A_n) =a_0 \Leb(A_n),
\end{split}
\]
completing the proof.
\end{proof}

\begin{lemma} \label{lem:keyfact}
Let $\eps\in(0,\frac12\delta_0)$ be small as in Propositions~\ref{prop:eps} and~\ref{prop:eps2}.
There exist $c_1>0$ and $N\ge1$ such that 
\[
\Leb\Big(\bigcup_{i=0}^{N} \{R=n+i\}\Big)\ge c_1 \Leb(A_{n-1})
\quad\text{for all $n\ge1$}.
\]
\end{lemma}

\begin{proof}
Fix $\lambda\in(0,1)$, $L>1$, $0<\delta<\delta_1<\delta_0$ and $N_1\ge1$ as defined from the outset.
Recall that $C_3(3\delta)^\alpha<\frac12\delta_0$ and $C_4(L+1)\delta<\delta_0$.
Choose $N_2\ge1$ such that $\lambda^{N_2}<\eps/\delta_0$ and take $N=N_1+N_2$.

We claim that
\begin{itemize}
\item[(*)]
For all $z\in \Lambda$, there exists $i\in\{1,\dots,N_1\}$ such that
$\pi (\phi_{i+N_2} W^u_\eps(z)\cap \tcD_L)\supset \cD_L$.
\end{itemize}

Fix $z\in \Lambda$.
By the definition of $N_1$, there exists $1\le i\le N_1$ such that $d_M(\phi_{-i}p,\phi_{N_2}z)<\delta$.
Let $y\in \cD_L$. Then
\[
d_M(\phi_{-i}y,\phi_{N_2}z)\le 
d(\phi_{-i}y,\phi_{-i}p)+d_M(\phi_{-i}p,\phi_{N_2}z)\le d(y,p)+d_M(\phi_{-i}p,\phi_{N_2}z)<(L+1)\delta.
\]
Using the local product structure and choice of $\delta$, we can
define $x\in W^{cs}_{\delta_1}(\phi_{-i}y)\cap W^u_{\delta_1}(\phi_{N_2}z)$.
Then $\phi_ix\in W^{cs}_{\delta_1}(y)\subset\tcD_L$ and
$g_ix=\pi \phi_ix=y$.
Also,
\[
d(x, \phi_{N_2}z)\le C_4d_M(\phi_{-i}y,\phi_{N_2}z)
<C_4(L+1)\delta<\delta_0.
\]
By the definition of $N_2$,
\[
\phi_ix\in \phi_i W^u_{\delta_0}(\phi_{N_2}z)
\subset \phi_{i+N_2} W^u_\eps(z).
\]
Hence we obtain that
$y=\pi \phi_ix\in \pi(\phi_{i+N_2}W^u_\eps(z)\cap\hcD_L)$
proving (*).

Next, we claim that 
\begin{itemize}
\item[(**)] 
For all $z\in \phi_nA_{n-1}$, $n\ge1$, there 
exist $i\in\{0,\dots,N\}$ and $j$ such that $U_{n+i,j}^1\subset \phi_{-n}W^u_{\delta_0}(z)$.
\end{itemize}

To prove (**), define $V_\eps=\phi_{-n}W^u_\eps(z)$.
By Proposition~\ref{prop:eps2}(b),
$V_\eps\subset A_{n-1}^{(\eps)}$.
We now consider two possible cases.

Suppose first that $V_\eps\subset A_{n+i}$ for all $0\le i\le N$.
By claim (*), there exists $1\le i\le N=N_1+N_2$ such that 
\[
\pi(\phi_{n+i}V_\eps\cap \tcD_L)=\pi(\phi_iW^u_\eps(z)\cap \tcD_L)\supset \cD_L,
\]
while $V_\eps\subset A_{n+i-1}$ by assumption.
This means that $V_\eps\supset U_{n+i,j}^L$ for some $j$. Hence
\[
U_{n+i,j}^1\subset U_{n+i,j}^L\subset V_\eps\subset  \phi_{-n}W^u_{\delta_0}(z),
\]
  and we are done.

In this way, we reduce to the second case where there exists $0\le i\le N$ least such that
$V_\eps\not\subset A_{n+i}$.
Since $i$ is least, 
$V_\eps\subset A_{n+i-1}^{(\eps)}$. (The $\eps$ is required in case $i=0$.)
By Proposition~\ref{prop:eps2}(a),
$V_\eps\subset\{t_{n+i-1}\le 1\}$.
Hence
\[
\begin{split}
V_\eps\setminus A_{n+i}
& =(V_\eps\cap B_{n+i})\,\cup\,(V_\eps\cap \{R=n+i\})
\\ & \subset\{t_{n+i-1}\le 1,\,t_{n+i}\ge1\}\,\cup\, \{R=n+i\}
\subset\bigcup_j U_{n+i,j}^2.
\end{split}
\]
Since $V_\eps\setminus A_{n+i}\neq\emptyset$, 
this means that there exists $j$ so that
$V_\eps$ intersects $U_{n+i,j}^2$.
Hence we can choose $a_2\in W^u_\eps(z)\cap \phi_n U_{n+i,j}^2$.

Recall that $\phi_{n+i}U_{n+i,j}^m \subset \tcD_m$ 
and $g_{n+i}U_{n+i,j}^m = \cD_m$ for $m=1,2$.
In particular, $b_2=\phi_ia_2\in \tcD_2$ and $c_2=g_ia_2\in \cD_2$.

Let $c_1\in\cD_1$.
Then $d_M(c_1,b_2)\le d_M(c_1,c_2)+d_M(c_2,b_2)<3\delta+\delta_1<4\delta_1$. Hence,
using the local product structure and definition of $\delta_1$,  we can define $b_1\in W^{cs}_{\delta_0}(c_1)\cap W^u_{\delta_0}(b_2)$ and $a_1=\phi_{-i}b_1$.
Note that
\[
\phi_ia_r=b_r,\qquad \pi b_r=c_r, \quad r=1,2.
\]
Hence
\[
d(a_1,a_2)\le d(b_1,b_2)\le C_3d(c_1,c_2)^\alpha<C_3(3\delta)^\alpha<\tfrac12\delta_0,
\]
and so $d(a_1,z)\le d(a_1,a_2)+d(a_2,z)<\frac12\delta_0+\eps<\delta_0$.
It follows that
$a_1\in W^u_{\delta_0}(z)$ and thereby that
$c_1\in g_i(W^u_{\delta_0}(z)\cap \phi_{-i}\hcD_1)$.
This proves that $\cD_1\subset g_i(W^u_{\delta_0}(z)\cap \phi_{-i}\hcD_1)$.
Hence $U_{n+i,j}^1\subset g_{-(n+i)}\cD_1\subset \phi_{-n}W^u_{\delta_0}(z)$ verifying claim~(**).

 \vspace{1ex}
We are now in a position to complete the proof of the lemma.
Let $n\ge1$, and let $Z\subset \phi_nA_{n-1}$ be a maximal set of points such that
the balls $W^u_{\delta_0/2}(z)$ are disjoint for $z\in Z$.
If $x\in \phi_nA_{n-1}$, then $W^u_{\delta_0/2}(x)$ intersects at least one
$W^u_{\delta_0/2}(z)$, $z\in Z$, by maximality of the set $Z$. Hence
$\phi_nA_{n-1}\subset \bigcup_{z\in Z}W^u_{\delta_0}(z)$.
It follows that
\[
A_{n-1}\subset \bigcup_{z\in Z}\phi_{-n}W^u_{\delta_0}(z).
\]

Let $z\in Z$ and let $U_z=U_{n+i,j}^1$ be as in claim~(**).
In particular, $g_{n+i}U_z=\cD_1=W^u_\delta(p)$.
Also, $\Leb(\phi_{n+i}U_z)\le 
|D\phi_1|_\infty^{im}\Leb(\phi_nU_z)$
where $m=\dim E^u$.
Hence, by~\eqref{eq:bddpi}, 
\[
\frac{1}{\Leb(\phi_nU_z)}\le |D\phi_1|_\infty^{Nm}\frac{1}{\Leb(\phi_{n+i}U_z)}
\le C_3 |D\phi_1|_\infty^{Nm}\frac{1}{\Leb(W^u_\delta(p))}.
\]
By~\eqref{eq:bdd}, 
\[
\frac{\Leb(\phi_{-n}W^u_{\delta_0}(z))}{\Leb(U_z)}
\le C_1
\frac{\Leb(W^u_{\delta_0}(z))}{\Leb(\phi_nU_z)}
\le 
K,
\]
where $K=C_1C_3|D\phi_1|_\infty^{Nm}\,\dfrac{\sup_{y\in Y}\Leb(W^u_{\delta_0}(y))}{\Leb(W^u_\delta(p))}$.

Finally, the sets $U_z$ are connected components of $\bigcup_{0\le i\le N}\{R=n+i\}$ lying in distinct disjoint sets $\phi_{-n}W^u_{\delta_0}(z)$.
Hence 
\[
\Leb(A_{n-1})  
 \le \sum_{z\in Z} \Leb(\phi_{-n}W^u_{\delta_0}(z))
 \le K \sum_{z\in Z}\Leb(U_z) \le 
K\Leb\Big(\bigcup_{0\le i\le N}\{R=n+i\}\Big),
\]
as required.
\end{proof}

We can now complete the proof of Theorem~\ref{thm:induce}.

\begin{cor}  \label{cor:tails}
$\Leb(R>n)=O(\gamma^n)$ for some $\gamma\in(0,1)$.
\end{cor}

\begin{proof}
By Corollary~\ref{cor:facts2} and Lemma~\ref{lem:keyfact},
\[
\begin{split}
\Leb(R\ge n) & =\Leb(A_{n-1})+\Leb(B_{n-1})
\\ & \le (1+a_0)\Leb(A_{n-1})
\le d_2\Leb\Big(\bigcup_{i=0}^{N} \{R=n+i\}\Big)
\end{split}
\]
where $d_2=c_1^{-1}(1+a_0)$.
It follows that
\[
\begin{split}
d_2^{-1}\Leb(R\ge n) & \le \Leb(R=n)+\dots +\Leb(R=n+N)
\\ & =\Leb(R\ge n)-\Leb(R> n+N).
\end{split}
\]
Hence
\[
\Leb(R>n+N)\le (1-d_2^{-1})\Leb(R\ge n).
\]
In particular, $\Leb(R>kN)\le \gamma^{kN}$ with $\gamma=(1-d_2^{-1})^{1/N}$ and the result follows.
\end{proof}

\section{Exponential decay of correlations for flows}
\label{sec:expdecay}

In this section, we consider exponential decay of correlations for a class of uniformly hyperbolic skew product flows satisfying a uniform nonintegrability condition, generalising from $C^2$ flows as treated in~\cite{AGY06} to $C^{1+\alpha}$ flows. In doing so, we remove the restriction in~\cite{BaladiVallee05,AraujoM16} that unstable manifolds are one-dimensional.

The arguments are a straightforward combination of those 
in~\cite{AraujoM16,AGY06}. We follow closely the presentation in~\cite{AraujoM16}, with the focus on incorporating the ideas from~\cite{AGY06} where required.

Quotienting by stable leaves leads to a class of semiflows considered in 
Subsection~\ref{sec:semiflow}. The flows are considered in
Subsection~\ref{sec:flow}. 

The current section is completely independent from Section~\ref{sec:induce}, so overlaps in notation will not cause any confusion.

\subsection{$C^{1+\alpha}$ uniformly expanding semiflows}
\label{sec:semiflow}

Fix $\alpha\in(0,1)$. Let $Y\subset\R^m$ be an open ball\footnote{More generally, we could consider a John domain as in~\cite{AGY06} but the current setting suffices for our purposes.} in Euclidean space with Euclidean distance $d$. We suppose that $\diam Y=1$. Let $\Leb$ denote Lebesgue measure on $Y$.  Let $\cP$ be a countable partition $\bmod\, 0$ of $Y$ consisting of open sets. 

Suppose that $F:\bigcup_{U\in\cP}U\to Y$ is $C^{1+\alpha}$ on each $U\in\cP$ and maps $U$ diffeomorphically onto $Y$.
Let $\cH=\{h:U\to Y:U\in\cP\}$ denote the family of inverse branches, and let $\cH_n$  denote the inverse branches for $F^n$.
We say that $F$ is a \emph{$C^{1+\alpha}$ uniformly expanding map} if there exist constants $C_1\ge1$, $\rho_0\in(0,1)$ such that 
\begin{itemize}
\item[(i)] $|Dh|_\infty\le C_1\rho_0^n$ for all $h\in\cH_n$, $n\ge1$;
\item[(ii)] $|\log|\det Dh|\,|_\alpha\le C_1$ for all $h\in\cH$;
\end{itemize}
where $|\psi|_\alpha=\sup_{y\neq y'}|\psi(y)-\psi(y')|/d(y,y')^\alpha$.
Under these assumptions, it is standard~\cite{Aaronson} that there exists a unique $F$-invariant absolutely continuous measure $\mu$. The density $d\mu/d\Leb$ is $C^\alpha$, bounded above and below, and $\mu$ is ergodic and mixing.

We consider roof functions $r:\bigcup_{U\in\cP}U\to\R^+$  that are $C^1$ on partition elements $U$ with $\inf r>0$. 
Define the suspension $Y^r=\{(y,u)\in Y\times\R:0\le u\le r(y)\}/\sim$ where 
$(y,r(y))\sim(Fy,0)$. The suspension semiflow $F_t:Y^r\to Y^r$ is given by
$F_t(y,u)=(y,u+t)$ computed modulo identifications, with ergodic invariant probability measure $\mu^r=(\mu\times{\rm Lebesgue})/\bar r$ where $\bar r=\int_Yr\,d\mu$.
We say that $F_t$ is a \emph{$C^{1+\alpha}$ uniformly expanding semiflow} if 
$F$ is a $C^{1+\alpha}$ uniformly expanding map and 
we can choose $C_1$ from condition~(i)  and $\eps>0$ such that
\begin{itemize}
\item[(iii)] $|D(r\circ h)|_\infty \le C_1$ for all $h\in\cH$;
\item[(iv)] $\sum_{h\in\cH} e^{\eps|r\circ h|_\infty}|\det Dh|_\infty <\infty$.
\end{itemize}

Let $r_n=\sum_{j=0}^{n-1}r\circ F^j$ and define
\[
\psi_{h_1,h_2}=r_n\circ h_1-r_n\circ h_2 : Y\to\R,
\]
for $h_1,h_2\in\cH_n$. We require the following \emph{uniform nonintegrability} condition~\cite[Equation~(6.6)]{AGY06}:
\begin{itemize}
\item[(UNI)] There exists $E>0$ and $h_1,h_2\in \cH_{n_0}$, for some sufficiently large $n_0\ge1$, with the following property: There exists a continuous unit vector field $\ell:\R^m\to\R^m$ such that $|D\psi_{h_1,h_2}(y)\cdot\ell(y)|\ge E$ for all $y\in Y$.
\end{itemize}
(The requirement ``sufficiently large'' can be made explicit as in~\cite[Equations~(2.1) to~(2.3)]{AraujoM16}.)
From now on, $n_0$, $h_1$ and $h_2$ are fixed.

Define $F_\alpha(Y^r)$ to consist of $L^\infty$ functions
$v:Y^r\to \R$ such that
$\|v\|_\alpha=|v|_\infty+|v|_\alpha<\infty$ where
\[
|v|_\alpha=\sup_{(y,u)\neq(y',u)}\frac{|v(y,u)-v(y',u)|}{d(y,y')^\alpha}.
\]
Define $F_{\alpha,k}(Y^r)$ to consist of functions
with $\|v\|_{\alpha,k}=\sum_{j=0}^k \|\partial_t^jv\|_\alpha<\infty$ where $\partial_t$ denotes differentiation along the semiflow direction.

We can now state the main result in this section.
Given $v\in L^1(Y^r)$, $w\in L^\infty(Y^r)$, define the correlation function
\[
\rho_{v,w}(t)=\int v\,w\circ F_t\,d\mu^r
-\int v\,d\mu^r \int w\,d\mu^r.
\]

\begin{thm} \label{thm:semiflow}
Suppose that $F_t:Y^r\to Y^r$ is a $C^{1+\alpha}$ uniformly expanding semiflow satisfying (UNI).
Then there exist constants $c,C>0$ such that
\[
|\rho_{v,w}(t)|\le C e^{-c t}\|v\|_{\alpha,1}\|w\|_{\alpha,1},
\]
for all $t>0$ and all $v,\,w\in F_{\alpha,1}(Y^r)$
(alternatively 
all $v\in F_{\alpha,2}(Y^r)$, $w\in L^\infty(Y^r)$).
\end{thm}

In the remainder of this subsection, we prove Theorem~\ref{thm:semiflow}.

For $s\in\C$, let $P_s$ denote the (non-normalised) transfer operator
\[
P_s=\sum_{h\in\cH}A_{s,h},
\qquad
A_{s,h}v=e^{-sr\circ h}|\det Dh|\,v\circ h .
\]

For $v:Y\to\C$, define $\|v\|_\alpha=\max\{|v|_\infty,|v|_\alpha\}$ where
$|v|_\alpha=\sup_{y\neq y'}|v(y)-v(y')|/d(y,y')^\alpha$.
Let $C^\alpha(Y)$ denote the space of functions $v:Y\to\C$ with $\|v\|_\alpha<\infty$.
We introduce the family of equivalent norms
\[
\|v\|_b=\max\{|v|_\infty, |v|_\alpha/(1+|b|^\alpha)\}, \quad b\in\R.
\]

\begin{prop} \label{prop:P}
Write $s=\sigma+ib$.
There exists $\eps\in(0,1)$ such that
the family $s\mapsto P_s$ of operators on $C^\alpha(Y)$ is continuous on
$\{\sigma>-\eps\}$.
Moreover, $\sup_{|\sigma|<\eps}\|P_s\|_b<\infty$.
\end{prop}

\begin{proof}
The first five lines of the proof of~\cite[Proposition~2.5]{AraujoM16} should be changed to the following:

Using the inequality $1-t\le -\log t$ valid for $t>0$, we obtain
for $a>b>0$ that
\(
a-b=a(1-\frac{b}{a})\le -a\log\frac{b}{a}=a(\log a-\log b).
\)
Hence
$\big||\det Dh(x)|-|\det Dh(y)|\big|\le |\det Dh|_\infty\big(\log|\det Dh(x)|-\log|\det Dh(y)|\big)$ and so by~(ii),
\begin{equation} \label{eq:h'0}
\big||\det Dh(x)|-|\det Dh(y)|\big|\le C_1|\det Dh|_\infty\, d(x,y)^\alpha \quad\text{for all
$h\in\cH$, $x,y\in Y$.}
\end{equation}
The proof now proceeds exactly as for~\cite[Proposition~2.5]{AraujoM16} (with
$R$, $h'$ and $|x-y|$ changed to 
$r$, $\det Dh$ and $d(x,y)$).
\end{proof}

The unperturbed operator $P_0$ has a simple leading eigenvalue $\lambda_0=1$
with strictly positive $C^\alpha$ eigenfunction $f_0$.   By Proposition~\ref{prop:P}, there exists $\eps\in(0,1)$ such that $P_\sigma$ has a continuous family of simple eigenvalues $\lambda_\sigma$
for $|\sigma|<\eps$ with associated $C^\alpha$ eigenfunctions $f_\sigma$.
For $s=\sigma+ib$ with $|\sigma|\le\eps$,
we define the normalised transfer operators
\[
L_sv=(\lambda_\sigma f_\sigma)^{-1}P_s(f_\sigma v)
=(\lambda_\sigma f_\sigma)^{-1}\sum_{h\in\cH}A_{s,h}(f_\sigma v).
\]
In particular, $L_\sigma 1=1$ and $|L_s|_\infty\le1$.

Set $C_2=C_1^2/(1-\rho)$, $\rho=\rho_0^\alpha$.  Then
\begin{itemize}
\item[(ii$_1$)]
$|\log |\det Dh|\,|_\alpha\le C_2$ for all $h\in\cH_n$, $n\ge1$,
\item[(iii$_1$)]
$|D(r_n\circ h)|_\infty\le C_2$ for all $h\in\cH_n$, $n\ge1$.
\end{itemize}

Write
\[
L_s^nv=\lambda_\sigma^{-n} f_\sigma^{-1}\sum_{h\in\cH_n}A_{s,h,n}(f_\sigma v), \quad
A_{s,h,n}v=e^{-sr_n\circ h}|\det Dh|v\circ h.
\]

\begin{lemma}[Lasota-Yorke inequality] \label{lem:LY}
There is a constant $C_3>1$ such that
\[
|L_s^nv|_\alpha\le C_3(1+|b|^\alpha) |v|_\infty+C_3\rho^n |v|_\alpha
\le C_3(1+|b|^\alpha)\{ |v|_\infty+\rho^n \|v\|_b\},
\]
for all $s=\sigma+ib$, $|\sigma|<\eps$, and all $n\ge1$, $v\in C^\alpha(Y)$.
\end{lemma}

\begin{proof}
It follows from (ii$_1$) that 
\[
\big||\det Dh(x)|-|\det Dh(y)|\big|\le C_2|\det Dh|_\infty\, d(x,y)^\alpha
\le C_2e^{C_2}|\det Dh(z)|\, d(x,y)^\alpha
\]
for all $h\in\cH_n$, $n\ge1$, $x,y,z\in Y$.
The proof now proceeds exactly as for~\cite[Lemma~2.7]{AraujoM16}.
\end{proof}

\begin{cor} \label{cor:LY}
$\|L_s^n\|_b\le 2C_3$ for all
$s=\sigma+ib$, $|\sigma|<\eps$, and all $n\ge1$.
\end{cor}

\begin{proof}
This is unchanged from~\cite[Corollary~2.8]{AraujoM16}.
\end{proof}

Given $b\in\R$, we define the cone
\begin{align*}
\cC_b=\Bigl\{ & \;(u,v): u,v\in C^\alpha(Y),\;u>0,\;0\le |v|\le u,\;
|\log u|_\alpha\le C_4|b|^\alpha,\\
& \qquad\qquad |v(x)-v(y)|\le C_4|b|^\alpha u(y)d(x,y)^\alpha\quad\text{for
all $x,y\in Y$} \;\Bigr\}.
\end{align*}
Throughout $B_\delta(y)=\{x\in \R^m:d(x,y)<\delta\}$.

\begin{lemma}[Cancellation Lemma] \label{lem:cancel}
Assume that the (UNI) condition is satisfied (with associated constants $E>0$ and $n_0\ge1$).
Let $h_1,h_2\in \cH_{n_0}$ be the branches from (UNI).

There exists $0<\delta<\Delta=4\pi/E$ such that for all
$s=\sigma+ib$, $|\sigma|<\eps$, $|b|\ge1$,
and all $(u,v)\in\cC_b$ we have the following:

For every $y'\in Y$ with $B_{(\delta+\Delta)/|b|}(y')\subset Y$, there exists $y''\in B_{\Delta/|b|}(y')$ such that
one of the following inequalities holds
on $B_{\delta/|b|}(y'')$:
\begin{description}
\item[Case $h_1$:]
$|A_{s,h_1,n_0}(f_\sigma v)+ A_{s,h_2,n_0}(f_\sigma v)| \le
\tfrac34 A_{\sigma,h_1,n_0}(f_\sigma u)+
A_{\sigma,h_2,n_0}(f_\sigma u)$,
\item[Case $h_2$:]
$|A_{s,h_1,n_0}(f_\sigma v)+ A_{s,h_2,n_0}(f_\sigma v)| \le
A_{\sigma,h_1,n_0}(f_\sigma u)+
\tfrac34 A_{\sigma,h_2,n_0}(f_\sigma u)$.
\end{description}
\end{lemma}

\begin{proof}
Let $\theta=V-b\psi_{h_1,h_2}$ where $\psi_{h_1,h_2}=r_{n_0}\circ h_1-r_{n_0}\circ h_2$ and $V=\arg(v\circ h_1)-\arg(v\circ h_2)$.

We follow the following steps from~\cite[Lemma~2.9]{AraujoM16}:

\begin{itemize}

\parskip = -2pt
\item[(1)]  Reduce to the situation where
$|v(h_my')|> \frac12 u(h_my')$ for both $m=1$ and $m=2$.

\item[(2)] Establish the estimate
$|V(y)-V(y')|\le \pi/6$
for all $y\in B_{(\delta+\Delta)/|b|}(y')$.

\item[(3)] Construct 
$y''\in B_{\Delta/|b|}(y')$ such that
\[
b(\psi_{h_1,h_2}(y'')-\psi_{h_1,h_2}(y'))=\theta(y')-\pi\, \bmod 2\pi.
\]

\item[(4)] Deduce that $|\theta(y)-\pi|\le 2\pi/3$
for all $y\in B_{(\delta+\Delta)/|b|}(y')$.

\item[(5)] Conclude the desired result.
\end{itemize}

Only step (3) requires any change from the argument in~\cite[Lemma~2.9]{AraujoM16}. We provide here the modified argument.
Approximate the continuous unit vector field $\ell:\R^m\to\R^m$ in (UNI) by a smooth vector field $\ell:\R^m\to\R^m$ with $|\ell(x)|\le 1$ for all $x\in\R^m$.
By condition (iii$_1$), the approximation can be chosen close enough that
\begin{equation} \label{eq:UNI}
|D\psi_{h_1,h_2}(y)\cdot\ell(y)|\ge \tfrac12 E\quad\text{for all $y\in Y$.}
\end{equation}

Let $g:[0,\Delta/|b|]\to \R^m$ be the solution to the initial value problem 
\[
\tfrac{dg}{dt}=\ell\circ g, \quad g(0)=y'
\]
and set $y_t=g(t)$.
Note that $d(y_t,y')\le \int_0^t|\ell(g(s))|\,ds\le \Delta/|b|$, so $y_t\in B_{\Delta/|b|}(y')$ for all $t\in[0,\Delta/|b|]$.
By the mean value theorem applied to
$\psi_{h_1,h_2}\circ g:[0,\Delta/|b|]\to\R$ and~\eqref{eq:UNI},
\[
|\psi_{h_1,h_2}(y_t)-\psi_{h_1,h_2}(y')|\ge t\inf_{s\in[0,\Delta/|b|]}|D\psi_{h_1,h_2}(y_s)\cdot\ell(y_s)|
\ge \tfrac12 Et=(2\pi/\Delta)t
\]
for all $t\in [0,\Delta/|b|]$.
It follows that $b(\psi_{h_1,h_2}(y_t)-\psi_{h_1,h_2}(y'))$ fills out an interval around $0$ of length at least $2\pi$ as~$t$ varies in $[0,\Delta/|b|]$.
In particular, we can choose $y''\in B_{\Delta/|b|}(y')$ such that (3) holds.
\end{proof}

Let $\{y_1',\dots,y_k'\}\subset Y$ be a maximal set of points such that 
the open balls $B_{(\delta+\Delta)/|b|}(y_i')$ are disjoint and contained in $Y$.

Let $(u,v)\in\cC_b$. For each $i=1,\dots,k$, there exists a ball
$B_i=B_{\delta/|b|}(y_i'')$ on which the conclusion of Lemma~\ref{lem:cancel} holds.
Write $\type(B_i)=h_m$ if we are in case $h_m$.
Let $\hB_i=B_{\frac12\delta/|b|}(y_i'')$

There exists a universal constant $C>0$ and a $C^1$ function $\omega_i:Y\to[0,1]$ such that
$\omega_i\equiv1$ on $\hB_i$, $\omega_i\equiv0$ on
$Y\setminus B_i$, and $\|\omega_i\|_{C^1}\le C|b|/\delta$.
Define $\omega:Y\to[0,1]$,
\[
\omega(y) = \begin{cases} \sum_{\type(B_i)=h_m}\omega_i(F^{n_0}y), & y\in\range{h_m},\,m=1,2 \\
0 & \text{otherwise}.
\end{cases}
\]
Note that $\|\omega\|_{C^1}\le C'|b|$ where $C'=C\delta$ is independent of $(u,v)\in\cC_b$ and $s\in\C$, and we can assume that $C'>4$.
Then $\chi=1-\omega/C':Y\to[\frac34,1]$ satisfies $|D\chi|\le |b|$.
Moreover, if $\type(B_i)=h_m$ then $\chi\equiv\eta$ on $\range h_m$
where $\eta=1-1/C'\in(0,1)$.

\begin{cor} \label{cor:cancel}
Let $\delta$, $\Delta$ be as in Lemma~\ref{lem:cancel}.
Let $|b|\ge1$, $(u,v)\in\cC_b$.
Let $\chi=\chi(b,u,v)$ be the $C^1$ function described above (using the
branches $h_1,h_2\in\cH_{n_0}$ from (UNI)).
Then $|L_s^{n_0}v|\le 
L_\sigma^{n_0}(\chi u)$ for all $s=\sigma+ib$, $|\sigma|<\eps$.
\end{cor}

\begin{proof}
This is immediate from Lemma~\ref{lem:cancel} and the definition of $\chi$.
\end{proof}

Define the disjoint union $\hB=\bigcup\hB_i$. 

\begin{prop} \label{prop:fed}
Let $K>0$. There exists $c_1>0$ such that
$\int_{\hB}w\,d\mu\ge c_1 \int_Y w\,d\mu$
for all $C^\alpha$ function $w:Y\to(0,\infty)$ with $|\log w|_\alpha\le K|b|^\alpha$, for all $|b|\ge 16\pi/E$.
\end{prop}

\begin{proof}
Let $y\in Y$. Since $(\delta+\Delta)/|b|\le 2\Delta/|b|=8\pi/(E|b|)\le\frac12$,
there exists $z\in Y$ with $B_{(\delta+\Delta)/|b|}(z)\subset Y$ such that $d(z,y)<(\delta+\Delta)/|b|$.
By maximality of the set of points $\{y'_1,\dots,y'_k\}$,
there exists $y'_i$ such that $B_{(\delta+\Delta)/|b|}(z)$ intersects
$B_{(\delta+\Delta)/|b|}(y'_i)$.
 Hence 
$Y\subset \bigcup_{i=1}^k B_i^*$
where $B_i^*= B_{3(\delta+\Delta)/|b|}(y_i')$.
Since the density $d\mu/d\Leb$ is bounded above and below, there is a constant $c_2>0$ such that
$\mu(\hB_i)\ge c_2 \mu(B_i^*)$ for each $i$.

Let $x\in \hB_i$, $y\in B_i^*$.
Then $d(x,y)\le 4(\delta+\Delta)/|b|$
and so $|w(x)/w(y)|\le e^{K'}$ where
$K'= \{4(\delta+\Delta)\}^\alpha K$.
It follows that
\[
\int_{\hB_i}w\,d\mu  
\ge\mu(\hB_i)\inf_{\hB_i} w \ge
 c_2 e^{-K'} \mu(B_i^*)\sup_{B_i^*} w
\ge c_1 \int_{B_i^*}w\,d\mu,
\]
where $c_1=c_2e^{-K'}$.
Since the sets $\hB_i\subset Y$ are disjoint,
\[
\int_{\hB} w\,d\mu  
=\sum_i\int_{\hB_i}w\,d\mu  
\ge c_1 \sum_i\int_{B_i^*}w\,d\mu
\ge c_1 \int_Y w\,d\mu
\]
as required.
\end{proof}

\begin{lemma}[Invariance of cone] \label{lem:cone}
There is a constant $C_4$ depending only on $C_1$, $C_2$, $|f_0^{-1}|_\infty$
and $|f_0|_\alpha$ such that the following holds:
 
For all $(u,v)\in \cC_b$, we have that
\[
\bigl(\,L_\sigma^{n_0}(\chi u)\,,\,L_s^{n_0}v\,\bigr)\in \cC_b,
\]
for all $s=\sigma+ib$, $|\sigma|<\eps$, $|b|\ge1$.
(Here, $\chi=\chi(b,u,v)$ is from Corollary~\ref{cor:cancel}.)
\end{lemma}

\begin{proof}
This is unchanged from~\cite[Lemma~2.12]{AraujoM16}.
\end{proof}

\begin{lemma}[$L^2$ contraction] \label{lem:L2}
There exist $\eps,\beta\in(0,1)$ such that
\[
\int_Y |L_s^{mn_0}v|^2\,d\mu\le \beta^m|v|_\infty^2
\]
for all $m\ge1$, $s=\sigma+ib$, $|\sigma|<\eps$, $|b|\ge\max\{16\pi/E,1\}$, and all $v\in C^\alpha(Y)$ satisfying $|v|_\alpha\le C_4|b|^\alpha|v|_\infty$.
\end{lemma}

\begin{proof}
Define
$u_0\equiv1, v_0=v/|v|_\infty$ and inductively,
\[
u_{m+1}=L_\sigma^{n_0}(\chi_mu_m), \qquad v_{m+1}=L_s^{n_0}(v_m),
\]
where $\chi_m=\chi(b,u_m,v_m)$.
It is immediate from the definitions that $(u_0,v_0)\in\cC_b$, and it follows
from Lemma~\ref{lem:cone} that $(u_m,v_m)\in\cC_b$ for all $m$.
Hence inductively the $\chi_m$ are well-defined as in Corollary~\ref{cor:cancel}.

We proceed as in~\cite[Lemma~2.13]{AraujoM16} in the following steps.

\begin{itemize}

\parskip=-2pt
\item[(1)] It suffices to show that there exists $\beta\in(0,1)$ such that
$\int_Y u_{m+1}^2\,d\mu\le\beta\int u_m^2\,d\mu$ for all $m$.

\item[(2)] Define $w=L_0^{n_0}(u_m^2)$. Then
 \[
u_{m+1}^2(y)\le \begin{cases}
\xi(\sigma)\eta_1 w(y) & y\in\hB \\
\xi(\sigma) w(y) & y\in Y\setminus\hB 
\end{cases}
\]
where $\xi(\sigma)$ can be made as close to $1$ as desired by shrinking $\eps$.
Here, $\eta_1\in(0,1)$ is a constant independent of $v$, $m$, $s$, $y$.

\item[(3)]
The function $w:Y\to\R$ satisfies the hypotheses of Proposition~\ref{prop:fed}; 
consequently $\int_{\hB}w\,d\mu\ge c_1 \int_{Y\setminus\hB}w\,d\mu$.
This leads to the desired conclusion.
\end{itemize}

\vspace{-5ex}
\end{proof}

\begin{lemma}[$C^\alpha$ contraction] \label{lem:dolg}
Let $E'=\max\{16\pi/E,2\}$.
There exists $\eps\in(0,1)$, $\gamma\in(0,1)$ and $A>0$
such that $\|P_s^n\|_b\le \gamma^n$ for
all $s=\sigma+ib$, $|\sigma|<\eps$, $|b|\ge E'$, $n\ge A\log |b|$.
\end{lemma}

\begin{proof}
This is unchanged from~\cite[Proposition~2.14, Corollary~2.15 and Theorem~2.16]{AraujoM16}.
\end{proof}

\begin{pfof}{Theorem~\ref{thm:semiflow}}
This is identical to~\cite[Section~2.7]{AraujoM16}.
We note that there is a typo in the statement of~\cite[Lemma~2.23]{AraujoM16}
where $|b|\le D'$ should be $|b|\ge D'$ (twice).
Also, for the second statement of~\cite[Proposition~2.18]{AraujoM16} it would be more natural to argue that 
\begin{align*}
\int_Y e^{\eps r}\,d\Leb
 & =\sum_{h\in\cH}\int_{h(Y)}e^{\eps r}\,d\Leb
\\ & =\sum_{h\in\cH}\int_Y e^{\eps r\circ h}|\det Dh|\,d\Leb
 \le \Leb(Y)\sum_{h\in\cH} e^{\eps |r\circ h|_\infty}|\det Dh|_\infty
\end{align*}
 which is finite by condition~(iv).
Hence $\int_Y e^{\eps r}\,d\mu<\infty$ 
by boundedness of $d\mu/d\Leb$.~
\end{pfof}

\subsection{$C^{1+\alpha}$ uniformly hyperbolic skew product flows}
\label{sec:flow}

Let $X=Y\times Z$ where $Y$ is an open ball of diameter $1$ with Euclidean metric $d_Y$ and $(Z,d_Z)$ is a compact Riemannian manifold.
Define the metric $d((y,z),(y',z'))=d_Y(y,y')+d_Z(z,z')$ on $X$.
Let $f(y,z)=(Fy,G(y,z))$ where $F:Y\to Y$, $G:X\to Z$ are $C^{1+\alpha}$.

We say that $f:X\to X$ is a \emph{$C^{1+\alpha}$ uniformly hyperbolic skew product}
if 
	$F:Y\to Y$ is a $C^{1+\alpha}$ uniformly expanding map satisfying conditions~(i) and~(ii) as in Section~\ref{sec:semiflow}, with absolutely continuous invariant probability measure $\mu$, and moreover
\begin{itemize}
	\item[(v)]  There exist constants $C>0$, $\gamma_0\in(0,1)$ such that
$d(f^n(y,z),f^n(y,z'))\le C\gamma_0^n d(z,z')$ for all $y\in Y$, $z,z'\in Z$.
\end{itemize}

Let $\pi^s:X\to Y$ be the projection $\pi^s(y,z)=y$.  This defines a semiconjugacy between $f$ and $F$ and there is a unique $f$-invariant ergodic probability measure $\mu_X$ on $X$ such that
$\pi^s_*\mu_X=\mu$.

Suppose that $r:\bigcup_{U\in\cP}U\to\R^+$ is $C^1$ on partition elements $U$ with
$\inf r>0$.  Define $r:X\to\R^+$ by setting
$r(y,z)=r(y)$.
Define the suspension $X^r=\{(x,u)\in X\times\R:0\le u\le r(x)\}/\sim$
where $(x,r(x))\sim(fx,0)$.  The suspension flow
$f_t:X^r\to X^r$ is given by $f_t(x,u)=(x,u+t)$ computed modulo identifications,
 with ergodic invariant probability measure 
$\mu_X^r=(\mu_X\times{\rm Lebesgue})/\bar r$.

We say that $f_t$ is a \emph{$C^{1+\alpha}$ uniformly hyperbolic skew product flow} provided
$f:X\to X$ is a $C^{1+\alpha}$ uniformly hyperbolic skew product as above, and
$r:Y\to\R^+$ satisfies conditions~(iii) and~(iv) as in Section~\ref{sec:semiflow}.
If $F:Y\to Y$ and $r:Y\to\R^+$ satisfy condition (UNI) from Section~\ref{sec:semiflow}, then we say that the skew product flow $f_t$ satisfies (UNI).

Define $F_\alpha(X^r)$ to consist of $L^\infty$ functions
$v:X^r\to \R$ such that 
$\|v\|_\alpha=|v|_\infty+|v|_\alpha<\infty$ where
\[
|v|_\alpha=\sup_{(y,z,u)\neq(y',z',u)}\frac{|v(y,z,u)-v(y',z',u)|}{d((y,z),(y',z'))^\alpha}.
\]
Define $F_{\alpha,k}(X^r)$ to consist of functions
with $\|v\|_{\alpha,k}=\sum_{j=0}^k \|\partial_t^jv\|_\alpha<\infty$ where $\partial_t$ denotes differentiation along the flow direction.

We can now state the main result in this section.
Given $v\in L^1(X^r)$, $w\in L^\infty(X^r)$, define the correlation function
\[
\rho_{v,w}(t)=\int v\,w\circ f_t\,d\mu_X^r
-\int v\,d\mu_X^r \int w\,d\mu_X^r.
\]

\begin{thm} \label{thm:flow}
Assume that $f_t:X\to X$ is a $C^{1+\alpha}$ hyperbolic skew product flow
satisfying the (UNI) condition.
Then there exist constants $c,C>0$ such that
\[
|\rho_{v,w}(t)|\le C e^{-c t}\|v\|_{\alpha,1}\|w\|_{\alpha,1},
\]
for all $t>0$ and all 
$v,\,w\in F_{\alpha,1}(X^r)$ (alternatively all
$v\in F_{\alpha,2}(X^r)$, $w\in F_\alpha(X^r)$). 
\end{thm}

\begin{proof}
This is unchanged from~\cite[Section~4]{AraujoM16}.
\end{proof}

\section{Proof of Theorem~\ref{thm:A}}
\label{sec:proof}

We return to the situation of Section~\ref{sec:induce}, so $\Lambda\subset M$ is a uniformly hyperbolic attractor for a $C^{1+\alpha}$ flow, $\alpha\in(0,1)$, defined on a compact Riemannian manifold.
Define the open unstable disk $Y=W^u_\delta(p)$ 
with discrete return time $R:Y\to\Z^+$ and induced map $F=\pi\circ \phi_R:Y\to Y$ as 
in Theorem~\ref{thm:induce}.

Under smoothness assumptions on holonomies, we verify the conditions on the suspension flow $f_t$ in Section~\ref{sec:expdecay} and obtain Theorem~\ref{thm:A} as an easy consequence.

\begin{prop} \label{prop:app1}
Suppose that the centre-stable holonomies are $C^{1+\alpha}$. (In particular,
$\pi:\hcD\to\cD$ is $C^{1+\alpha}$.)
Then (after shrinking $\delta_0$ in Section~\ref{sec:induce} if necessary) $F$ is a $C^{1+\alpha}$ uniformly expanding map.
\end{prop}

\begin{proof}
As in Remark~\ref{rmk:induce}, it is immediate that $F|_U:U\to Y$ is a $C^{1+\alpha}$ diffeomorphism for all $U\in\cP$.
Let $h:Y\to U$ be an inverse branch with
$R|_U=n$, and define $\pi_U=\pi|_{\phi_n(U)}:\phi_n(U)\to\cD$.  Then 
\[
\lambda^{-1}|v|\le \lambda^{-n}|v|\le |D\phi_n(x)v|\le |(D\pi_U)^{-1}|_\infty |DF(x)v|
\]
for all $x\in U$, $v\in T_xY$. 
Hence $|Dh|_\infty\le \rho_0$ where $\rho_0=\lambda\sup_U|(D\pi_U)^{-1}|_\infty$.
Shrinking $\delta_0$, we can ensure that $\rho_0<1$.
In particular, condition~(i) in Section~\ref{sec:semiflow} holds (with $C_1=1$).
Condition~(ii) is the standard distortion estimate.
\end{proof}

In the remainder of this section, we suppose moreover that the stable holonomies are $C^{1+\alpha}$. 
Shrink $\delta_0\in(0,1)$ as in Proposition~\ref{prop:app1} and
shrink $\delta_1\in(0,\delta_0)$ so that 
$\phi_t(W^s_{\delta_1}(y))\subset W^s_{\delta_0}(\phi_ty)$ for all $t>0$, $y\in\Lambda$.
Recall that $\cD=W^u_{\delta_0}(p)$ and
\[
\hcD=\bigcup_{y\in\cD}W^{cs}_{\delta_0}(y)
=\bigcup_{|t|<\delta_0}\phi_t \Big(
\bigcup_{y\in \cD}W^s_{\delta_0}(y)\Big).
\]
The projection $\pi^s:\bigcup_{y\in \cD}W^s_{\delta_0}(y)\to \cD$ given by $\pi^s|W^s_{\delta_0}(y)\equiv y$ is
$C^{1+\alpha}$. Moreover, 
$\pi=\pi^s\circ \phi_{r_0}$ where 
$\phi_{r_0}:\hcD\to\bigcup_{y\in \cD}W^s_{\delta_0}(y)$ and $r_0:\hcD\to (-\delta_0,\delta_0)$ is $C^{1+\alpha}$.
Define $r=R+r_0$ on $Y$. 
The choice $\delta_0<1$ ensures that $\inf r\ge 1-\delta_0>0$.
Define the corresponding semiflow $F_t:Y^r\to Y^r$.

\begin{prop} \label{prop:app2}
$F_t:Y^r\to Y^r$ is a $C^{1+\alpha}$ uniformly expanding semiflow.
\end{prop}

\begin{proof}
By Proposition~\ref{prop:app1}, $F$ is a $C^{1+\alpha}$ uniformly expanding map. In particular, conditions~(i) and~(ii) are satisfied.

Notice that
$F=\pi^s\circ \phi_r$ where $r=R+r_0$ is $C^{1+\alpha}$ on partition elements $U\in\cP$. 
Since $Dr=Dr_0$ on partition elements, it is immediate that $\sup_{h\in\cH}|D(r\circ h)|_\infty
\le |D r_0|_\infty\sup_{h\in\cH}|Dh|_\infty\le \rho_0|Dr_0|_\infty<\infty$
verifying condition~(iii) on~$r$.
Recall that $\Leb(R>n)=O(\gamma^n)$ for some $\gamma\in(0,1)$, so we can choose $\eps>0$ such that $\int_Y e^{\eps R}\,d\Leb<\infty$.
Condition (ii) ensures that
$|\det Dh|_\infty\le (\Leb Y)^{-1}e^{C_1}\Leb(\range h)$ for all $h\in\cH$. Hence 
$\sum_{h\in\cH}e^{\eps |r\circ h|_\infty}|\det Dh|_\infty\ll 
\sum_{h\in\cH}e^{\eps |R\circ h|_\infty}\Leb(\range h)=\int_Y e^{\eps R}\,d\Leb<\infty$
verifying condition~(iv) on $r$.
\end{proof}

We now make a $C^{1+\alpha}$ change of coordinates so that $\hcD$ is identified with $\cD\times W_{\delta_0}^{s}(p)\times(-\delta_0,\delta_0)$ where $\{y\}\times W_{\delta_0}^{s}(p)$
is identified with $W_{\delta_0}^{s}(y)$ for all $y\in\cD$ and $(-\delta_0,\delta_0)$ is the flow direction.
Let $X=Y\times Z$ where $Z= W_{\delta_0}^{s}(p)$ and define $r:X\to(0,\infty)$
by $r(y,z)=r(y)$. Also, define $f=\phi_r:X\to X$ 
and the corresponding suspension flow $f_t:X^r\to X^r$ 

\begin{prop} \label{prop:app3}
$f_t:X^r\to X^r$ is a $C^{1+\alpha}$ uniformly hyperbolic skew product flow.
\end{prop}

\begin{proof}
Note that $\pi^s(X)=Y$ and $\pi^s(y,z)=y$.
Also, $f(y,z)=(Fy,G(y,z))$ where
$G:X\to Z$ is $C^{1+\alpha}$. Since $Z$ corresponds to the exponential contracting stable foliation, condition~(v) in Section~\ref{sec:flow} is satisfied.
Hence $f:X\to X$ is a $C^{1+\alpha}$ uniformly hyperbolic skew product
and the corresponding suspension flow $f_t:X^r\to X^r$ 
is a $C^{1+\alpha}$ uniformly hyperbolic skew product flow.
\end{proof}

Next we recall the standard argument that joint nonintegrability implies (UNI) in the current situation. 
(Similar arguments are given for instance in~\cite[Section~3]{AMV15} and~\cite[Section~5.3]{M18}.)

Joint nonintegrability is defined in terms of the temporal distortion function.
To define this intrinsically (independently of the inducing scheme) we have to introduce the first return time $\tau:X\to\R^+$ and the Poincar\'e map
$g:X\to X$ given by
\[
\tau(x)=\inf\{t>0:\phi_t(x)\in X\},
\qquad
g(x)=\phi_{\tau(x)}(x).
\]
Note that $\tau$ is constant along stable leaves by the choice of $X$.

For $x_1,x_2\in X$, define the local product
$[x_1,x_2]$ to be the unique intersection point of $W^u(x_1)\cap W^s(x_2)$.
The \emph{temporal distortion function} $D$ is defined to be
\[
D(x_1,x_2)
  =\sum_{j=-\infty}^\infty \big\{ \tau(g^jx_1)
-\tau(g^j[x_1,x_2]) -\tau(g^j[x_2,x_1])+\tau(g^jx_2)\big\}
\]
at points $x_1,x_2\in X$.
The stable and unstable bundles are jointly integrable if and only if $D\equiv0$.

\begin{lemma} \label{lem:D}
Joint nonintegrability of the stable and unstable bundles implies (UNI).
\end{lemma}

\begin{proof}
For points $x,x'\in X$ with $x'\in W^u(x)$, we define
\[
D_0(x,x')=\sum_{j=1}^\infty \big\{\tau(g^{-j}x)-\tau(g^{-j}x')\big\}.
\]
Since $\tau$ is constant along stable leaves,
\begin{align*}
D(x_1,x_2)
  &=\sum_{j=1}^\infty \big\{ \tau(g^{-j}x_1)
-\tau(g^{-j}[x_1,x_2]) -\tau(g^{-j}[x_2,x_1])+\tau(g^{-j}x_2)\big\}
\\
&= D_0(x_1,[x_1,x_2])+D_0(x_2,[x_2,x_1]).
\end{align*}

Next, we find a more convenient expression for $D_0$ in terms of $r$ and $f$.
Note that for any $x\in X$, there exists $N(x)\in \Z^+$ (the number of returns to $X$ up to time $r(x)$) such that 
\[
r(x)=\sum_{\ell=0}^{N(x)-1}\tau(g^\ell x), \qquad  f(x)=g^{N(x)}x.
\]

Corresponding to the partition $\cP$ of $Y$, we define the collection $\tcP=\{\bar U\times \bar Z:U\in\cP\}$ of closed subsets of $X$.
Suppose that $x,x'\in V_0$, $V_0\in\tcP$, with $x'\in W^u(x)$.
The induced map $f:X\to X$ need not be invertible since it is not the first return to $X$.
However, we may construct suitable inverse branches 
$z_j$, $z_j'$ of $x$, $x'$ as follows.
Set $z_0=x$, $z_0'=x'$.
Since $f$ is transitive and continuous on closures of partition elements,
there exists $V_1\in\tcP$ and $z_1\in V_1$ such that $fz_1=z_0$.
Since $F$ is full-branch, $f(W^u(z_1)\cap V_1)\supset W^u(z_0)$, so there exists $z_1'\in W^u(z_1)\cap V_1$ such that $fz_1'=z_0'$.
Inductively, we obtain $V_n\in\tcP$ and $z_j,z_j'\in V_n$ with $z_j'\in W^u(z_j)$ such that $fz_j=z_{j-1}$ and $fz_j'=z_{j-1}'$.

By construction, $z_{j-1}=fz_j=g^{N(z_j)}z_j$.
Hence $z_j=g^{-(N(z_1)+\dots+N(z_j))}x$
and 
\[
r(z_j) 
=\sum_{\ell=0}^{N(z_j)-1}\tau(g^\ell g^{-(N(z_1)+\dots+N(z_j))}x)
=\sum_{\ell=N(z_1)+\dots+N(z_{j-1})+1}^{N(z_1)+\dots+N(z_j)}\tau(g^{-\ell}x).
\]
A similar expression holds for $r(z_j')$.
Hence
\[
D_0(x,x')=
 \sum_{j=1}^\infty \big\{ r(z_j)
-r(z_j')\big\}.
\]

We are now in a position to complete the proof of the lemma, showing that if 
(UNI) fails, then $D\equiv0$.
To do this, we make use of~\cite[Proposition~7.4]{AGY06} (specifically the equivalence of their conditions~1 and~3). Namely, the failure of the (UNI) condition in Section~\ref{sec:semiflow} means that
we can write $r=\xi\circ F-\xi+\zeta$ on $Y$ where $\xi:Y\to\R$ is continuous (even $C^1$) and
$\zeta$ is constant on partition elements $U\in\cP$.
Extending $\xi$ and $\zeta$ trivially to $X=Y\times Z$, we obtain that
$r=\xi\circ f-\xi+\zeta$ on $X$ where $\xi:X\to\R$ is continuous and constant on stable leaves, and
$\zeta$ is constant on elements $V\in\tcP$.
In particular,
\[
\sum_{j=1}^n r(z_j)=\sum_{j=1}^n \big\{\xi(z_{j-1}-\xi(z_j)+\zeta(z_j)\big\}
=\xi(x)-\xi(z_n)+\sum_{j=1}^n \zeta(z_j).
\]
For $x,x'\in V_0$, $V_0\in\tcP$, with $x'\in W^u(x)$, 
it follows that
\[
\sum_{j=1}^n \big\{ r(z_j) -r(z_j')\}
 =\xi(x)-\xi(x')- \xi(z_n) +\xi(z_n').
\]
Taking the limit as $n\to\infty$, we obtain that
$D_0(x,x')=\xi(x)-\xi(x')$.
Hence
$D(x_1,x_2)=\xi(x_1)-\xi([x_1,x_2])-\xi([x_2,x_1])+\xi(x_2)$.
Since $\xi$ is constant on stable leaves, $D(x_1,x_2){=0}$ as required.
\end{proof}

\begin{pfof}{Theorem~\ref{thm:A}} 
By Proposition~\ref{prop:app3} and Lemma~\ref{lem:D}, $f_t$ is a $C^{1+\alpha}$ uniformly hyperbolic flow satisfying (UNI).
The result for $C^{1+\alpha}$ observables follows from Theorem~\ref{thm:flow}.
As in~\cite{Dolgopyat98a}, the result follows from a standard interpolation argument (see also~\cite[Corollary~2.3]{AraujoM16}).
\end{pfof}

 \paragraph{Acknowledgements}
This research was supported in part by CMUP, which is financed by national funds through FCT - Funda\c c\~ao para a Ci\^encia e a Tecnologia, I.P., under the project with reference UIDB/00144/2020. PV also acknowledges financial support from the project PTDC/MAT-PUR/4048/2021 and from the grant CEECIND/03721/2017 of the Stimulus of Scientific Employment, Individual Support 2017 Call, awarded by FCT.
The authors are also grateful to Funda\c c\~ao Oriente for the financial support during the Pedro Nunes Meeting in Convento da Arr\'abida,
 where part of this work was done.

\end{document}